\documentclass[12pt]{amsart}
\pdfoutput=1
\usepackage{fullpage}
\usepackage{hyperref}
\usepackage{amssymb}
\usepackage{amstext}
\usepackage{amsmath}
\usepackage{amsthm}
\usepackage{stmaryrd}
\usepackage{verbatim}
\usepackage{enumerate}
\usepackage{mathtools}
\usepackage{booktabs, array}

\newtheorem{theorem}{Theorem}

\newtheorem{proposition}[theorem]{Proposition}
\newtheorem{lemma}[theorem]{Lemma}
\theoremstyle{remark}

\newcommand{\SL}{\operatorname{SL}}

\newcommand{\D}{\operatorname{D}}

\begin{document}
\title{Eichler-Selberg relations for the third-order mock theta functions}
\author{Wei Wang}
\address{Department of Mathematics, Shaoxing University, Shaoxing 312000, China}
\email{weiwang\_math@163.com}

\begin{abstract}
Ramanujan's third-order mock theta function $f(q)$ admits the well-known Appell-Lerch series representation:
\[
\sum_{n=0}^{\infty}\frac{q^{n^2}}{(-q;q)_n^2}=\frac{2}{(q;q)_{\infty}}\sum_{n=-\infty}^{\infty}\frac{(-1)^nq^{\frac{3}{2}n^2+\frac{1}{2}n}}{1+q^n}.
\]
In this paper, we establish a natural generalization of this classical identity by utilizing the theory of harmonic Maass forms. Furthermore, we prove analogous Eichler-Selberg type relations for another third-order mock theta function $\omega(q)$. The method presented in this paper can be extended to study other classes of mock theta functions.
\end{abstract}
\thanks{}
\maketitle
\section{Statement of the results}
In his famous letter to Hardy, Ramanujan introduced 17  new series that  he
called mock theta functions, and he grouped these series by order. 
Since then, mock theta functions have attracted considerable interest, see for example \cite{zbMATH03952860,zbMATH05206073,zbMATH01336943,zbMATH02526376}. 
Furthermore, since Zwegers' systematic investigation of the modularity of mock theta functions in his doctoral thesis \cite{zbMATH05784497}, these series can be understood within the framework of harmonic Maass forms, leading to the proof of the well-known conjectures regarding mock theta functions (e.g. \cite{zbMATH05051175,zbMATH05374823}). A comprehensive survey of this topic can be found in \cite{zbMATH06828732}.

An example of a third-order mock theta function is given by
 \begin{equation}\label{equation_mocktheta_f}
 f(q)=\sum_{n=0}^{\infty}c_f(n)q^n:=\sum_{n=0}^{\infty}\frac{q^{n^2}}{(-q;q)_n^2}.
 \end{equation}
Throughout this paper we use the standard notation
\[
(z;q)_n:=\left\{\begin{aligned} 
&\prod_{j=0}^{n-1}(1-zq^{j})~~\text{   if }n\in\mathbb{N},\\
&\prod_{j=0}^{\infty}(1-zq^{j})~~\text{   if }n=\infty,
\end{aligned}
\right.
\]
where $q:=e^{2\pi i\tau}$ for $\tau \in \mathbb{H}$ the upper half plane. Watson \cite{zbMATH02526376} studied the behaviour of the third order functions under the modular group, and he found the following Appell-Lerch series representation of $f(q)$:
\begin{equation}\label{equation_Watson}
f(q)=\frac{2}{(q;q)_{\infty}}\sum_{n=-\infty}^{\infty}\frac{(-1)^nq^{\frac{3}{2}n^2+\frac{1}{2}n}}{1+q^n}.
\end{equation}
By Euler's pentagonal number theorem
\[
(q;q)_{\infty}=\sum_{k\in\mathbb{Z}}(-1)^kq^{\frac{k(3k+1)}{2}},
\]
Watson's identity (\ref{equation_Watson}) can be written as
\begin{equation}\label{equation_Watson2}
\sum_{n\geq 0}\sum_{k\in\mathbb{Z}}(-1)^k c_f\left(n-\frac{k(3k+1)}{2}\right)q^n=2\sum_{n=-\infty}^{\infty}\frac{(-1)^nq^{\frac{3}{2}n^2+\frac{1}{2}n}}{1+q^n}.
\end{equation}
where we define $c_f(n)=0$ for $n<0$. The aim of this paper is to provide a natural generalization of (\ref{equation_Watson2}). To state our results, we first introduce some notation. Let $T_n(x)$ denote the Chebyshev polynomials of the first kind, which are defined by the following recurrence relation:
\begin{align*}
&T_0(x)=1,~T_1(x)=x,\\
&T_{n+1}(x)=2xT_n(x)-T_{n-1}(x) \text{ for }n\geq 1.
\end{align*}
They are given by the explicit formula:
\begin{equation}\label{equation_che}
T_n(x)=\sum_{j=0}^{\lfloor \frac{n}{2}\rfloor}\binom{n}{2j}(x^2-1)^jx^{n-2j}.
\end{equation}
Let $E_{n}$ denote the Euler numbers, which are defined by the following generating function:
\[
\frac{2}{e^{t}+e^{-t}}=\sum_{n\geq 0}E_n\frac{t^n}{n!}.
\]
The following is one of the main results of this paper.
\begin{theorem}\label{theorem_main}
For every $v\geq 1$, the $q$-series
\begin{equation}\label{equation_maintheorem}
\begin{aligned}
&6^v\cdot\sum_{n\geq 1}\sum_{k\in\mathbb{Z}}n^{v}T_{2v}\left(\frac{6k+1}{\sqrt{24n}}\right)(-1)^kc_f\left(n-\frac{k(3k+1)}{2}\right)q^n\\
&-2\sum_{n\in\mathbb{Z}}\left(\frac{12}{2n+1}\right)\frac{n^{2v}q^{\frac{n(n+1)}{6}}}{1+q^n}\\
&+\frac{2}{E_{2v}}\sum_{n\geq 1}\frac{(-1)^{n-1}(2n-1)^{2v}q^{2n-1}}{1+q^{2n-1}}-\frac{2^{2v+1}}{E_{2v}}\sum_{n\geq 1}\frac{(-1)^{n-1}n^{2v}q^{n}}{1+q^{2n}}
\end{aligned}
\end{equation}
is a cusp form in $S_{2v+1}(\Gamma_0(8),\left(\tfrac{-1}{\cdot}\right))$. In particular, when $v=1$ or $2$, this cusp form vanishes.
\end{theorem} 
Our results are in the spirit of the Eichler-Selberg class number formula (see \cite[Chapter 6.3]{zagier1991modular} or \cite[(1.3)]{zbMATH07957369}):
\[
\mathrm{Tr}(n; 2k) = -\frac{1}{2} \sum_{r \in \mathbb{Z}} n^{k-1}U_{2k-2}\left(\frac{r}{2\sqrt{n}}\right)H(4n-r^2) - \frac{1}{2} \sum_{d \mid n} \min(d, n/d)^k,
\]
where $H(d)$ denotes the Hurwitz–Kronecker class numbers and $U_{2k-2}(x)$ denotes the Chebyshev polynomials of the second kind.

The $q$-series in the second line of \eqref{equation_maintheorem} can be regarded as a proper generalization of the right-hand side of \eqref{equation_Watson2}. Indeed, we observe that
\[
\sum_{n\in\mathbb{Z}}\frac{(-1)^nq^{\frac{3}{2}n^2+\frac{1}{2}n}}{1+q^n}=\sum_{n\in\mathbb{Z}}\left(\frac{12}{2n+1}\right)\frac{q^{\frac{n(n+1)}{6}}}{1+q^n}.
\]
For a proof of this identity, see Lemma \ref{lemma_q0}. The $q$-series in the third line is essentially the Lambert series representation of the Eisenstein series on $\Gamma_0(8)$ with the character $\left(\tfrac{-1}{\cdot}\right)$. Remarkably, although our proof assumes $v\geq 1$, since the Eisenstein parts cancel out at weight one, the result remains consistent for $v=0$. Watson's identity \eqref{equation_Watson} can be viewed as a special case of our result at weight one, providing a hint for generalizing the approach to other mock theta functions.

For $v=1$, our result yields the identity
\begin{align*}
&\sum_{n\geq 1}\sum_{k\in\mathbb{Z}}\left((6k+1)^2-12n\right)(-1)^kc_f\left(n-\frac{k(3k+1)}{2}\right)q^n\\
=&4\sum_{n\in\mathbb{Z}}\left(\frac{12}{2n+1}\right)\frac{n^{2}q^{n(n+1)/6}}{1+q^n}
+4\sum_{n\geq 1}\frac{(-1)^{n-1}(2n-1)^{2}q^{2n-1}}{1+q^{2n-1}}-16\sum_{n\geq 1}\frac{(-1)^{n-1}n^{2}q^{n}}{1+q^{2n}},
\end{align*}
and for $v=2$, we get
\begin{align*}
&\sum_{n\geq 1}\sum_{k\in\mathbb{Z}}\left[(6k+1)^4-24n(6k+1)^2+72n^2\right](-1)^kc_f\left(n-\frac{k(3k+1)}{2}\right)q^n\\
=&4\sum_{n\in\mathbb{Z}}\left(\frac{12}{2n+1}\right)\frac{n^{4}q^{n(n+1)/6}}{1+q^n}
-\frac{4}{5}\sum_{n\geq 1}\frac{(-1)^{n-1}(2n-1)^{4}q^{2n-1}}{1+q^{2n-1}}+\frac{64}{5}\sum_{n\geq 1}\frac{(-1)^{n-1}n^{4}q^{n}}{1+q^{2n}}.
\end{align*}
To the best of the author's knowledge, the two identities presented above have not previously appeared in the literature. For $v=3$, the cusp form in Theorem \ref{theorem_main} is non-vanishing. Numerical computations show that this cusp form is given by
\[
-\frac{20736}{61}\eta^4(z)\eta^{10}(2z),
\]
where  
\[
\eta(z)=q^{-1/24}\prod_{n=1}^{\infty}(1-q^n)=q^{-1/24}(q;q)_{\infty}
\]
denotes the Dedekind eta function.

Let
\[
\omega(q):=\sum_{n\geq 0}c_{\omega}(n)q^n=\sum_{n=0}^{\infty}\frac{q^{2n(n+1)}}{(q;q^2)_{n+1}^2}
\]
be another third order mock theta function. We establish the following analogous result.
\begin{theorem}\label{theorem_main2}
For every $v\geq 1$, the $q$-series
\begin{align*}
\sum_{n\geq 0}&\sum_{k\in\mathbb{Z}}(12n+9)^{v}T_{2v}\left(\frac{6k+1}{\sqrt{12n+9}}\right)(-1)^kc_{\omega}(n-k(3k+1))q^{\frac{n}{2}+\frac{3}{8}}\\
&+\sum_{n\in \mathbb{Z}}\left(\frac{12}{n+2}\right)\frac{n^{2v}q^{\frac{n(n+4)}{24}}}{1-q^{\frac{n}{2}}}\\
&-\frac{1}{2E_{2v}}\sum_{n\geq 1}\frac{(2n-1)^{2v}q^{\frac{n}{4}-\frac{1}{8}}}{1+q^{\frac{n}{2}-\frac{1}{4}}}+\frac{1}{2E_{2v}}\sum_{n\geq 1}\frac{(-1)^{n-1}(2n-1)^{2v}q^{\frac{n}{4}-\frac{1}{8}}}{1-q^{\frac{n}{2}-\frac{1}{4}}}
\end{align*}
is a cusp form in $S_{2v+1}(\Gamma^{0}(8),\chi_{-4})$. In particular, when $v=1$ or 2, this cusp form vanishes.
\end{theorem}
This can be seen as a generalization of the classical identity given by Watson \cite{zbMATH02526376}:
\[
\omega(q)=\frac{1}{(q^2;q^2)_{\infty}}\sum_{n\in\mathbb{Z}}(-1)^nq^{3n(n+1)}\frac{1+q^{2n+1}}{1-q^{2n+1}}.
\]
The second line of the $q$-series in Theorem \ref{theorem_main2} also serves as the proper generalization of the summation on the right-hand side of the identity above, see Lemma \ref{lemma_q1}.

The primary strategy of this paper is outlined below. This approach is also applicable to a wide class of other mock theta functions. First, we identify the non-holomorphic part and the modular transformation properties of the given mock theta function. Second, we interpret the Appell-Lerch or Hecke-type representation (see \cite{CHEN2020107037}) of a mock theta function can be seen as the holomorphic projection of a weight 1/2 eta-quotient multiplied by the corresponding harmonic Maass form. Third, we consider the Rankin-Cohen brackets of this eta-quotient and the harmonic Maass form. Its holomorphic projection yields a modular form on a certain congruence subgroup with specific characters. The Eisenstein part is then determined by analyzing its expansions at various cusps. 

\section{Preliminaries}
\subsection{Notations}
In this section, we fix some standard notation. For $\tau\in \mathbb{H}$ the upper half plane, $q:=e^{2\pi i\tau}$. For positive integer $n\in\mathbb{Z}$, let $\zeta_n:=e^{\frac{2\pi i}{n}}$ denote the primitive $n$-th root of unity. Let $\chi$ be a Dirichlet character. For any matrix $\gamma=\begin{pmatrix}
a&b\\
c&d
\end{pmatrix}\in \operatorname{GL}^{+}_2(\mathbb{Z})$, define $\chi(\gamma):=\chi(d)$. Let $k$ be a positive integer. For a function $f(\tau)$ defined on the upper half plane $\mathbb{H}$,  the weight $k$ slash operator is defined by
\[
(f|_k\gamma)(\tau)=(ad-bc)^{k/2}j(\gamma,\tau)^{-k}f\left(\frac{a\tau+b}{c\tau+d}\right),~j(\gamma,\tau):=c\tau+d.
\]
The incomplete gamma function is defined by
\[
\Gamma(t,x):=\int_{x}^{\infty}e^{-z}z^{t-z}dz.
\] 

Let $\Gamma$ be  a subgroup of finite index of $\SL_2(\mathbb{Z})$ and $f,g$ be two modular forms of weight $k$. Then the Petersson scalar product of $f$ and $g$ with respect to the group $\Gamma$ is defined by
\[
\langle f,g\rangle=\frac{1}{[\SL_2(\mathbb{Z}):\Gamma]}\int_{\Gamma\setminus \mathbb{H}}f(\tau)\overline{g(\tau)}y^{k-2}dxdy,~\tau=x+iy,
\]
provided this integral exists.

We recall the Rankin–Cohen brackets, which serve as a generalization of the standard multiplication of modular forms. Let $f$ and $g$ be two smooth functions defined on the upper half-plane and let $k$ and $l$ be two integers or half-integers. The $v$-th Rankin–Cohen bracket for $f$ and $g$ is defined as
\[
[f,g]_{v}:=\sum_{0\leq j\leq v} (-1)^j \binom{k+v-1}{v-j} \binom{l+v-1}{j}\D^{j}f\D^{v-j}g, 
\]
where $\D f:=\frac{1}{2\pi i}\frac{d f}{d\tau}$. The 0-th Rankin–Cohen bracket of two modular forms is the product of these two modular forms. The following result given by Cohen \cite[Theorem 7.1]{MR0382192} shows that Rankin-Cohen bracket, in a sense, is a generalization of the product map:
\begin{equation}\label{equation_Rankin-Cohen}
[f\big|_k\gamma,g\big|_l\gamma]_{v}=[f,g]_{v}\big|_{k+l+2v}\gamma,~\forall \gamma\in\SL_2(\mathbb{Z}).
\end{equation}

\subsection{Mock theta functions}
Zwegers generalized Watson's work and provided new insights into the study of mock theta functions by demonstrating that these series, when augmented by a suitable non-holomorphic completion, yield functions that satisfy modular transformation properties. These functions are now known as harmonic Maass forms and have since been the subject of extensive research. We recall Zwegers' completion of the series (\ref{equation_mocktheta_f}), defining
\begin{align*}
&g_0(\tau):=\sum_{n\in\mathbb{Z}}(-1)^n(n+1/3)q^{\tfrac32(n+\tfrac13)^2},\\
&g_1(\tau):=-\sum_{n\in\mathbb{Z}}(-1)^n(n+1/6)q^{\tfrac32(n+\tfrac16)^2},\\
&g_2(\tau):=\sum_{n\in\mathbb{Z}}(n+1/3)q^{\tfrac32(n+\tfrac13)^2}.
\end{align*}
Set
\[
F(\tau):=\left(q^{-1/24}f(q),2q^{1/3}\omega(q^{1/2}),2q^{1/3}\omega(-q^{1/2})\right)^{T} \text{ and }G(\tau):=\left(g_1,g_0,-g_2\right)^{T}.
\]
Zwegers \cite[Theorem 3.6]{zbMATH01724355} showed that
\[
H(\tau):=F(\tau)-2i\sqrt{3}\int_{-\overline{\tau}}^{i\infty}\frac{G(z)}{(-i(\tau+z))^{\tfrac 12}}dz
\]
is a vector valued real-analytic modular form of weight 1/2, transforming as 
\begin{equation}\label{equation_trans-H}
H(\tau+1)=\begin{pmatrix}
\zeta^{-1}_{24} &0 &0\\
0 &0 &\zeta_3\\
0 &\zeta_3 &0\\
\end{pmatrix}H(\tau),~~
H(-1/\tau)=\zeta_8^{-1}\begin{pmatrix}
0 &1 &0\\
1 &0 &0\\
0 &0 &-1\\
\end{pmatrix}H(\tau)\tau^{\frac{1}{2}}.
\end{equation}

The following result establishes the modular transformation law for the Rankin–Cohen bracket of the Dedekind eta function and the required harmonic Maass form.
\begin{lemma}
We have, for $\gamma\in\Gamma_0(8)$,
\[
[H_1(\tau),\eta(\tau)]_v\big|_{1+2v}\gamma=\left(\frac{-4}{d_{\gamma}}\right)[H_1(\tau),\eta(\tau)]_v,
\]
where $d_{\gamma}$ denotes the lower-right entry of the matrix $\gamma$, and $H_1(z)$ denotes the first component of the vector $H$.
\end{lemma}
\begin{proof}
We first recall the transformation properties property of the Dedekind eta function:
\begin{equation}\label{equation_eta-mod}
\eta(\tau+1)=\zeta_{24}\eta(\tau),~~~\eta(-1/\tau)=\zeta_8^{-1}\eta(\tau)\tau^{\frac{1}{2}}.
\end{equation}

The congruence subgroup $\Gamma_0(8)$ is generated by the following four matrices:
\[
-I = \begin{pmatrix} -1 & 0 \\ 0 & -1 \end{pmatrix}, \quad
T = \begin{pmatrix} 1 & 1 \\ 0 & 1 \end{pmatrix}, \quad
\gamma_{1/4} = \begin{pmatrix} 5 & -1 \\ 16 & -3 \end{pmatrix}, \quad
\gamma_{1/2} = \begin{pmatrix} 5 & -2 \\ 8 & -3 \end{pmatrix}.
\]
The modular curve $X_0(8)$ associated with the congruence subgroup $\Gamma_0(8)$ has genus $g = 0$ and $c = 4$ cusps, represented by the equivalence classes of $\infty, 0, 1/2$, and $1/4$. Since $\Gamma_0(8)$ contains no elliptic elements, its image in $\text{PSL}_2(\mathbb{Z})$ is a free group of rank $2g + c - 1 = 3$. Consequently, $\Gamma_0(8)$ can be minimally generated by $-I$ along with three independent parabolic elements stabilizing the cusps. The matrices $T$, $\gamma_{1/4}$, and $\gamma_{1/2}$ act as the parabolic generators stabilizing the cusps $\infty$, $1/4$, and $1/2$, respectively. 

Since the group $\SL_2(\mathbb{Z})$ is generated by $T$ and $S:=\begin{pmatrix}
0 & -1 \\ 1 & 0
\end{pmatrix},$
we express the matrix $\gamma_{1/4}$ and $\gamma_{1/2}$ as words in the generators $T$ and $S$:
\[
\gamma_{1/4}=-S T^{-3}  S  T^5  S\text{ and }\gamma_{1/2}=-S  T^{-2} S T^3 S.
\]
By the modular transformation properties of $H$ \eqref{equation_trans-H}, we find that
\[
H_1(\tau)\big|_{1/2}\gamma_{1/4}=H_1(\tau)\big|_{1/2}(-S T^{-3}S  T^5  S)=\zeta_4^{-1}\zeta_8^{-1}\zeta_3^{-3}(-\zeta_8)^{-1}\zeta_3^5\zeta_8^{-1}H_1(\tau)=\zeta_{24}^{13}H_1(\tau).
\]
The modular transformation properties of the Dedekind eta function \eqref{equation_eta-mod} implies that
\[
\eta(\tau)\big|_{1/2}\gamma_{1/4}=\zeta_4^{-1}\zeta_8^{-3}\zeta_{24}^2\eta(\tau)=\zeta_{24}^{-13}\eta(\tau).
\]
Using the properties of the Rankin–Cohen brackets in \eqref{equation_Rankin-Cohen}, we obtain
\[
[H_1(\tau),\eta(\tau)]_v\big|_{1+2v}\gamma_{1/4}=[H_1(\tau)\big|_{1+2v}\gamma_{1/4},\eta(\tau)\big|_{1+2v}\gamma_{1/4}]_v=[H_1(\tau),\eta(\tau)]_v.
\]
Following an analogous calculation, we have
\begin{align*}
&[H_1(\tau),\eta(\tau)]_v\big|_{1+2v}(-I)=-[H_1(\tau),\eta(\tau)]_v,~&[H_1(\tau),\eta(z)]_v\big|_{1+2v}T=[H_1(z),\eta(z)]_v,\\
&[H_1(\tau),\eta(\tau)]_v\big|_{1+2v}\gamma_{1/2}=[H_1(\tau),\eta(\tau)]_v.
\end{align*}
These relations characterize the transformation properties of $[H_1(\tau),\eta(\tau)]_v$ under the generators of $\Gamma_0(8)$, which completes the proof.
\end{proof}
\subsection{Holomorphic projection}
In the previous section, we obtained a real-analytic modular form $[H_1(\tau),\eta(\tau)]_v$. To study it within the framework of holomorphic modular forms, we require the tool of holomorphic projection. The holomorphic projection was first introduced by Sturm \cite{zbMATH03674175} and later extensively utilized by Gross and Zagier \cite{zbMATH03983362} in their seminal work on Heegner points. Imamo{\u{g}}lu, Raum, and Richter \cite{zbMATH06685253} provided a spectral-theoretic account of holomorphic projection and applied it to vector-valued harmonic Maass forms, particularly to mock theta functions. It provides a direct analytical bridge between real-analytic modular transformations and the arithmetic data encoded in the Fourier coefficients of holomorphic cusp forms.

Let $k > 2$ be an integer and $F: \mathbb{H} \to \mathbb{C}$ be a smooth function that transforms with weight $k$ under the congruence subgroup $\Gamma$ with the multiplier system $\nu$ and has the Fourier expansion
\[
F(z)=\sum_{n\in\mathbb{Q}}a_F(n,y)q^n,
\]
where $y=\Im (\tau)$. For every cusp $\kappa$ of $\Gamma$, fix $\gamma_{\kappa}\in\SL_2(\mathbb{Z})$ with $\gamma_{\kappa}(i\infty)=\kappa$. Assume that for some $\delta,\epsilon>0$, the function $F(z)$ satisfies
\begin{enumerate}
\item $\left(F\big|_k\gamma_{\kappa}\right)(z)=c_{0,\kappa}+O(y^{-\delta})$ for each cusp $\kappa$ of $\Gamma$,

\item $a_F(n,y)=O(y^{1-k+\epsilon})$ as $y\rightarrow 0$ for all $n>0$.
\end{enumerate}
The holomorphic projection operator, denoted by $\pi_{\mathrm{hol}}$, assigns to $F$ a unique holomorphic modular form $f = \pi_{\mathrm{hol}}(F) \in M_k(\Gamma,\nu)$ such that their Petersson inner products with any holomorphic cusp form coincide. Specifically, for all $g \in S_k(\Gamma)$, the projection satisfies the identity
\begin{equation}\label{equation_holoproj_inner}
    \langle g, \pi_{\mathrm{hol}}(F) \rangle = \langle g, F \rangle.
\end{equation}
From a computational perspective, the holomorphic projection can be determined explicitly via its Fourier coefficients. The $n$-th Fourier coefficient $a_n$ of the resulting holomorphic cusp form $f(z) = a_0+\sum_{n\in\mathbb{Q}^{+}} a_n q^n$ is given by the integral formula:
\[
    a_n = \frac{(4\pi n)^{k-1}}{(k-2)!} \int_{0}^{\infty} a_F(n,y) e^{-4\pi n y} y^{k-2} dy,
\]
and the constant term $a_0$  of $f(z)$ coincides with the constant term of $F(z)$ at the cusp $i\infty$. 

We also require the following properties of the holomorphic projection operator, which facilitate the computation of its expansion at various cusps.
\begin{proposition}\label{prop_holoproj-trans}
With the above notation, we have
\[
\pi_{\mathrm{hol}}(F\big|_k\gamma_{\kappa})=\pi_{\mathrm{hol}}(F)\big|_k\gamma_{\kappa}.
\]
\end{proposition}
\begin{proof}
Recall that the Poincar\'e series with respect to the cusp $\kappa$ is defined by
\[
P_k^n(\Gamma,\nu,\kappa;z)=\sum_{g\in\Gamma_{\kappa}\setminus \Gamma}\overline{\chi(g)}e^{2\pi i(n/w)\gamma_{\kappa}^{-1}gz}j(\gamma_{\kappa}^{-1}g,z)^{-k},
\]
where $\Gamma_{\kappa}=\Gamma\cap \gamma_{\kappa}\Gamma_{\infty}\gamma_{\kappa}^{-1}$ and $w$ denotes the width of the cusp $\kappa$. It has the property (see \cite[Chapter 8.2]{cohen2017modular}): 
\begin{enumerate}
\item For $f\in M_k(\Gamma,\nu)$ and $f\big|_k\gamma_{\kappa}=\sum_{n\geq 0}a_{\kappa}(n)q^{n/w}$,
\[
a_{\kappa}(n)=\frac{(4\pi n)^{k-1}}{(k-2)!}\frac{w^k}{\left[\SL_2(\mathbb{Z}):\Gamma\right]}\langle  f(z),P_k^n(\Gamma,\nu,\kappa;z) \rangle.
\]
\item If $\gamma\in\SL_2(\mathbb{Z})$, $P_k^n(\Gamma,\nu,\kappa;z)\big|_{k}\gamma =P_k^n(\gamma^{-1}\Gamma\gamma,\gamma\nu\gamma^{-1},\gamma^{-1}(\kappa);z)$.

\end{enumerate}

Now by (\ref{equation_holoproj_inner}) and the property of the Poincar\'e series, we get
\begin{align*}
\langle  \pi_{\mathrm{hol}}(F),P_k^n(\Gamma,\nu,\kappa)\rangle&=\langle  F,P_k^n(\gamma_{\kappa}^{-1}\Gamma\gamma_{\kappa},\gamma_{\kappa}\nu\gamma_{\kappa}^{-1},i\infty)\big|_{k}\gamma_{\kappa}^{-1}\rangle\\
&=\langle  \pi_{\mathrm{hol}}(F\big|_k\gamma_{\kappa}),P_k^n(\gamma_{\kappa}^{-1}\Gamma\gamma_{\kappa},\gamma_{\kappa}\nu\gamma_{\kappa}^{-1},i\infty)\rangle.
\end{align*}
Up to a common constant factor, the first term in the above expression is the $n$-th Fourier coefficient of $\pi_{\mathrm{hol}}(F)\big|_k\gamma_{\kappa}$, while the last term is the $n$-th Fourier coefficient of $\pi_{\mathrm{hol}}(F\big|_k\gamma_{\kappa})$. This concludes the proof of the proposition.
\end{proof}

The following result complements the results of Mertens \cite{zbMATH06620622}. Although this specific case was not addressed in \cite{zbMATH06620622}, the proof is analogous; we therefore provide only a brief sketch here. We fix an integer $w>0$ and an integer $\rho$ satisfying $0\leq \rho\leq w$. We assume that $f$ is a mock theta function, i.e. its shadow  is a (linear combination of) weight 3/2 unary theta series $\theta_{\psi}$ and  $g$ is also a weight 1/2 unary theta series $\theta_{\chi}$. In other words, we assume that $f$ and $g$ have the Fourier expansions
\[
f(\tau)=\sum_{n+\rho\geq 0}a_f^{+}(n)q^{\frac{n}{w}}+\sum_{n>0}\psi(n)n^{-\tfrac12}\Gamma\left(\frac12,4\pi\frac{n^2}{w}y\right)q^{-\frac{tn^2}{w}},
\]
and 
\[
g(\tau)=\sum_{n\in\mathbb{Z}}\chi(n)q^{\frac{\rho n^2}{w}},
\]
respectively. 
\begin{theorem}\label{theorem_holoproj}
With the above notation, the holomorphic projection of $[f,g]_v$ is given by 
\begin{align*}
&\frac{1}{2}\left(\rho/w\right)^v\binom{2v}{v}a_f^{+}(-\rho)+\binom{2v}{v}(4w)^{-v}\sum_{n\geq 1}\sum_{m\in\mathbb{Z}}T_{2v}\left(\sqrt{\frac{\rho}{n}}m\right)n^{v}\chi(m)a_{f}^{+}(n-\rho m^2)q^{\frac{n}{w}}\\
&+\sqrt{\pi}(4w)^{-v}\binom{2v}{v}\sum_{\rho m^2-tn^2=r}\chi(m)\psi(n)\left(\sqrt{\rho}m-\sqrt{t}n\right)^{2v}q^{\frac{r}{w}}.
\end{align*}
\end{theorem}
\begin{proof}
We decompose $f$ into its holomorphic and non-holomorphic parts: $f=f^{+}+f^{-}$. Accordingly, the holomorphic decomposition of $[f,g]_v$ is given by
\[
\pi_{\mathrm{hol}}([f,g]_v)=[f^{+},g]_v+\pi_{\mathrm{hol}}([f^{-},g]_v).
\]
The first part follows directly from the definition of the Rankin–Cohen bracket \eqref{equation_Rankin-Cohen}:
\[
[f^{+},g]_v=\sum_{j=0}^v(-1)^j\binom{v-\tfrac12}{v-j}\binom{v-\tfrac12}{j}\D^jf^{+}\D^{v-j}g.
\]
To simplify the coefficients, we employ the following identity for binomial coefficients involving half-integers:
\[
\binom{v-\tfrac12}{v-j}\binom{v-\tfrac12}{j}=2^{-2v}\binom{2v}{v}\binom{2v}{2j},
\]
which is a consequence of the duplication formula for the Gamma function, and we obtain
\begin{align*}
[f^{+},g]_v&=2^{-2v}\binom{2v}{v}\sum_{n\geq 0}\sum_{m\in\mathbb{Z}}\sum_{j=0}^{v}(-1)^j\binom{2v}{2j}\left(\frac{n}{w}\right)^j\left(\frac{\rho m^2}{w}\right)^{v-j}\chi(m)a_f^{+}(n)q^{\frac{n+\rho m^2}{w}}\\
&=2^{-2v}\binom{2v}{v}w^{-j}\sum_{n\geq 0}\sum_{m\in\mathbb{Z}}\sum_{j=0}^{v}\binom{2v}{2j}(\rho m^2-n)^j(\rho m^2)^{v-j}\chi(m)a_{f}^{+}(n-\rho m^2)q^{\frac{n}{w}}.
\end{align*}
The preceding expression can be reformulated in terms of Chebyshev polynomials \eqref{equation_che}, consistent with the statement of our result. 

Regarding the holomorphic projection of the non-holomorphic component, we invoke the results of Mertens \cite[Theorem 4.6]{zbMATH06620622} with necessary modifications, as we treat the case of weight $(\frac{1}{2},\frac{1}{2})$ not explicitly covered in \cite{zbMATH06620622}. If $f^{-}=\sum_{n>0}c^{-}_f(n)\Gamma(\tfrac12,4\pi\tfrac{n^2}{w}y)q^{-\frac{tn^2}{w}}$ and $g=\sum_{n\geq 0} a_g(n)q^{\frac{n}{w}}$, then
\[
\pi_{\mathrm{hol}}([f^{-},g]_v)=\sqrt{\pi}(4w)^{-v}\binom{2v}{v}\sum_{r\geq 1}\sum_{M-N=r}a_g(M)c^{-}_f(N)N^{-1/2}(M^{\frac{1}{2}}-N^{\frac{1}{2}})q^{r}.
\]
In our setting,  we have $c_f^{-}(N)=\psi(n)n^{-1/2}$ for $N=tn^2$ and $c_f^{-}(N)=0$ otherwise, $a_g(M)=\chi(m)$ for $M=\rho m^2$ and $a_g(M)=0$ otherwise. Consequently, in this case, $\pi_{\mathrm{hol}}([f^{-},g]_v)$ is given by
\[
\sqrt{\pi}(4w)^{-v}\binom{2v}{v}\sum_{r\geq 1}\sum_{\rho m^2-tn^2=r}\chi(m)\psi(n)\left(\sqrt{\rho}m-\sqrt{t}n\right)^{2v}q^{\frac{r}{w}},
\]
and the proof is completed.
\end{proof}
\section{Proof of Theorems}
\subsection{Proof of Theorem 1}
Using the definition of the incomplete gamma function, we have for $n\in\mathbb{N}$ that
\begin{equation}\label{equation_incomplete-gamma}
i(2\pi n)^{-\frac{1}{2}}\Gamma\left(\frac{1}{2},4\pi ny\right)q^{-n}=\int_{-\bar{\tau}}^{i\infty}\frac{e^{2\pi i nz}}{(-i(z+\tau))^{\frac{1}{2}}}dz,
\end{equation}
and 
\[
g_1(\tau)=-\frac{1}{6}\sum_{n\in\mathbb{Z}}\left(\frac{-12}{n}\right)nq^{\frac{n^2}{24}}
\]
then we get the Fourier expansion
\[
H_1(\tau)=q^{-1/24}f(q)-\frac{2}{\sqrt{\pi}}\sum_{n\geq 0}\left(\frac{-12}{n}\right)\Gamma\left(\frac{1}{2},\frac{\pi n^2y}{6}\right)q^{-\frac{n^2}{24}}.
\]
For $v\geq 1$, we apply Theorem \ref{theorem_holoproj} with $a_f^{+}(n)=c_f((n+1)/24)$ and $g(\tau)=\eta(\tau)$ and get
\begin{equation}\label{equation_holoproj-Fe}
\begin{aligned}
24^v\binom{2v}{v}^{-1}\pi_{\mathrm{hol}}([H_1,\eta]_v)&=\frac{1}{2}+6^v\cdot\sum_{n\geq 1}\sum_{k\in\mathbb{Z}}T_{2v}\left(\frac{6k+1}{\sqrt{24n}}\right)(-1)^kn^{v}c_f\left(n-\frac{k(3k+1)}{2}\right)q^n\\
&-2\sum_{m > n > 0} \left(\frac{12}{m}\right)\left(\frac{-12}{n}\right) \left(\frac{m-n}{2}\right)^{2v} q^{\frac{m^2-n^2}{24}}.
\end{aligned}
\end{equation}
We have shown that $[H_1,\eta]_v$ is a real-analytic  modular form on $\Gamma_0(8)$ with the Dirichlet character $\left(\frac{-4}{\cdot}\right)$. By the property of the holomorphic projection, the $q$-series \eqref{equation_holoproj-Fe} belongs to the space $M_{1+2v}(\Gamma_0(8),\left(\frac{-4}{\cdot}\right))$. In Lemma \ref{lemma_Hecke-f}, we transform the Hecke-type $q$-series in \eqref{equation_holoproj-Fe} into the generalized Appell–Lerch type series. Our aim is to determine the Eisenstein component of this modular form, as carried out in the sequel.

Let $\psi,\phi$ be two primitive Dirichlet characters modulo $N_1$ and $N_2$, respectively. Set $N=N_1N_2$, $k\geq 2$ with $\chi(-1)=(-1)^k$, the Eisenstein series of weight $k$ with respect to the character $\psi,\phi$ is defined by
\[
E_{k}^{\psi,\phi}(\tau)=\delta(\psi)L(1-k,\phi)+2\sum_{n=1}^{\infty}\sigma_{k-1}^{\psi,\phi}(n)q^n,~q=e^{2\pi i\tau},
\]
where
\[
\sigma_{k-1}^{\psi,\phi}(n):=\sum_{\substack{d\mid n\\d>0}}\psi(n/d)\phi(d)d^{k-1}.
\]
Weisinger \cite{weisinger1977some} showed that Eisenstein series are well-behaved under the action of Atkin-Lehner operators. In this context, we restrict our attention to a special case of these operators--the Fricke involution. 
\begin{proposition}[Weisinger, Proposition 1]\label{proposition_Fricke-invo}
Let $W_N:=\begin{pmatrix}
0 &-1\\
N&0
\end{pmatrix}$, then we have
\[
E_{k}^{\psi,\phi}\big|_{k}W_N=\psi(-1)\frac{\tau(\psi)}{\tau(\bar{\phi})}\left(\frac{N_2}{N_1}\right)^{k/2}E_k^{\bar{\phi},\bar{\psi}},
\]
where $\tau(\psi)$ and $\tau(\bar{\phi})$ denote the Gauss sums.
\end{proposition}

Let $\chi_{-4}:=\left(\frac{-4}{\cdot}\right)$. The Eisenstein subspace of  $M_{2v+1}(\Gamma_0(8),\chi_{-4})$ is spanned by four series:
\begin{equation}\label{equation_EisGamma08}
E_{2v+1}^{1,\chi_{-4}}(\tau),~E_{2v+1}^{1,\chi_{-4}}(2\tau),~E_{2v+1}^{\chi_{-4},1}(\tau),~E_{2v+1}^{\chi_{-4},1}(2\tau).
\end{equation}
Observe the following relationship between the special values of Dirichlet $L$-functions and the Euler numbers:
\[
L(1-k,\chi_{-4})=-\frac{B_{k,\chi_{-4}}}{k}=\frac{E_{k-1}}{2}.
\]

The modular curve $X_0(8)$ possesses four inequivalent cusps, namely: $\left\{i\infty,0,1/2,1/4\right\}$. For the cusp 0, we choose the matrix $\gamma_0=S$ and for the cusp $s=1/a, a=2$ or 4, we choose the matrix $\gamma_s=\begin{pmatrix}
1 & 0\\
a & 1
\end{pmatrix}$, such that $\gamma_s(i\infty)=s$. Using Proposition \eqref{proposition_Fricke-invo}, we calculate the Eisenstein series \eqref{equation_EisGamma08} under the slash operators $\big|_{2v+1}\gamma_s$ and get the values of \eqref{equation_EisGamma08} at the cusps. We summarize the results in the table below.
\begin{table}[htbp]
    \centering
    \caption{Values of Eisenstein series on $\Gamma_0(8)$ at cusps}
    \label{tab:eisenstein_cusps}
    \renewcommand{\arraystretch}{1.5}
    \setlength{\tabcolsep}{14pt}
    \begin{tabular}{l >{$}c<{$} >{$}c<{$} >{$}c<{$} >{$}c<{$}}
        \toprule
        & i\infty & 0 & \tfrac{1}{2} & \tfrac{1}{4} \\
        \midrule
        $E_{k}^{1,\chi_{-4}}(\tau)$       & E_{k-1}/2 & 0 & 0 & E_{k-1}/2 \\
        $E_{k}^{1,\chi_{-4}}(2\tau)$      & E_{k-1}/2 & 0 & 0 & 0 \\
        $E_{k}^{\chi_{-4},1}(\tau)$       & 0 & -2^{-2k}E_{k-1}\zeta_4  & 0 & 0 \\
        $E_{k}^{\chi_{-4},1}(2\tau)$      & 0 & -2^{-3k}E_{k-1}\zeta_4  & -2^{-2k}E_{2v}\zeta_4  & 0 \\
        \bottomrule
    \end{tabular}
\end{table}

To determine the Eisenstein component of $\pi_{\mathrm{hol}}([H_1,\eta]_v)$, we proceed to evaluate its values at the cusps. At the cusp $\frac{1}{2}$, noting that
\[
\begin{pmatrix}
1 &0\\
2 &1
\end{pmatrix}=-ST^{-2}S,
\]
it follows from \eqref{equation_trans-H} and \eqref{equation_eta-mod} that
\[
H_1\big|_{\frac{1}{2}}(-ST^{-2}S)=\zeta_4^{-1}\zeta_{8}^{-2}\zeta_3^{-2}H_1,~\eta\big|_{\frac{1}{2}}(-ST^{-2}S)=\zeta_4^{-1}\zeta_8^{-2}\zeta_{24}^{-2}\eta,
\]
then by \eqref{equation_Rankin-Cohen}, 
\[
[H_1,\eta]_{v}\big|_{1+2v}(-ST^{-2}S)=[F_1\big|_{1+2v}(-ST^{-2}S),\eta\big|_{1+2v}(-ST^{-2}S)]_v=\zeta_4[H_1,\eta]_{v}.
\]
Thus, by Proposition \ref{prop_holoproj-trans}, the values of $24^v\binom{2v}{v}^{-1}\pi_{\mathrm{hol}}([H_1,\eta]_v)$ at the $\frac{1}{2}$ is $\frac{\zeta_4}{2}$.  Similarly, we can evaluate the values of $24^v\binom{2v}{v}^{-1}\pi_{\mathrm{hol}}([H_1,\eta]_v)$ at the remaining cusps (for the cusp 0, see Section \ref{section_Proof of Theorem 2}), the results are summarized in the following table.
\begin{table}[htbp]
    \centering
    \caption{Values of $24^v\binom{2v}{v}^{-1}\pi_{\mathrm{hol}}([H_1,\eta]_v)$ at cusps}
    \label{tab:eisenstein_cusps2}
    \renewcommand{\arraystretch}{1.5}
    \setlength{\tabcolsep}{14pt}
    \begin{tabular}{l >{$}c<{$} >{$}c<{$} >{$}c<{$} >{$}c<{$}}
        \toprule
        & i\infty & 0 & \tfrac{1}{2} & \tfrac{1}{4} \\
        \midrule
        Series \eqref{equation_holoproj-Fe}       & \frac{1}{2} & 0 & \frac{\zeta_4}{2} & -\frac{1}{2} \\ 
        \bottomrule
    \end{tabular}
\end{table}

To determine the Eisenstein part of $24^v\binom{2v}{v}^{-1}\pi_{\mathrm{hol}}([F_1,\eta]_v)$, we let this part be
\[
a_1E_{2v+1}^{1,\chi_{-4}}(\tau)+a_2E_{2v+1}^{1,\chi_{-4}}(2\tau)+a_3E_{2v+1}^{\chi_{-4},1}(\tau)+a_4E_{2v+1}^{\chi_{-4},1}(2\tau).
\]
By considering the values of the function $24^v\binom{2v}{v}^{-1}\pi_{\mathrm{hol}}([F_1,\eta]_v)$ at various cusps, we obtain the following system of linear equations:
\[
\left\{
\begin{aligned}
&a_1E_{2v}+a_2E_{2v}=1,\\
&-2^{-2(2v+1)}E_{2v}a_3-2^{-3(1+2v)}E_{2v}a_4=0,\\
&-2^{-2(2v+1)}\zeta_4E_{2v} a_4=\frac{\zeta_4}{2},\\
&E_{2v}a_1=-1.
\end{aligned}
\right.
\]
By solving this system of linear equations, we obtain an explicit expression for the Eisenstein part of $24^v\binom{2v}{v}^{-1}\pi_{\mathrm{hol}}([F_1,\eta]_v)$:
\begin{equation}\label{equation_Eispart}
-\frac{1}{E_{2v}}E_{2v+1}^{1,\chi_{-4}}(\tau)+\frac{2}{E_{2v}}E_{2v+1}^{1,\chi_{-4}}(2\tau)+\frac{2^{2v}}{E_{2v}}E_{2v+1}^{\chi_{-4},1}(\tau)-\frac{2^{4v+1}}{E_{2v}}E_{2v+1}^{\chi_{-4},1}(2\tau)
\end{equation}
By Lemma \ref{lemma_Eis}, this is precisely the Lambert series appearing in the third line of \eqref{equation_maintheorem}, and the proof of Theorem \ref{theorem_main} is completed.

For $v=1$ and $v=2$, a finite computation of Fourier coefficients confirms that this cusp form vanishes identically. Indeed, by the valence formula, to verify the vanishing of a modular form, it suffices to check only the first few terms of its Fourier expansion (see \cite[Corollary 5.6.14]{cohen2017modular}).

\subsection{Proof of Theorem 2}\label{section_Proof of Theorem 2}
Zwegers' vector valued harmonic Maass form \eqref{equation_trans-H} implies that the mock theta function $\omega(q)$ is indeed the representation of $f(q)$ at the cusp 0. By  the integral representation of the incomplete Gamma function \eqref{equation_incomplete-gamma}, we obtain the Fourier expansion of the second component $H_2(\tau)$ of $H(\tau)$ as follows: 
\begin{equation}
H_2(\tau)=2q^{1/3}\omega(q^{1/2})+\frac{2}{\sqrt{\pi}}\sum_{n\geq 1}\left(\frac{n}{3}\right)(-1)^{n-1}\Gamma\left(\frac{1}{2},\frac{2\pi n^2 y}{3}\right)q^{-\frac{n^2}{6}}.
\end{equation}
Note that while the coefficients of the non-holomorphic part are not a single Dirichlet character, this does not preclude the application of Theorem \ref{theorem_holoproj}. Applying Theorem \ref{theorem_holoproj}, we obtain that
\begin{align*}
24^v\binom{2v}{v}^{-1}\pi_{\mathrm{hol}}([H_2,\eta]_v)&=24^v\binom{2v}{v}^{-1}[2q^{\frac13}\omega(q^{\frac{1}{2}}),\eta]_v\\
&+2\sum_{r\geq 1}\sum_{m^2-4n^2=r}\bigg(\frac{12}{m}\bigg) \bigg(\frac{n}{3}\bigg)(-1)^{n-1}\left(\frac{m-2n}{2}\right)^{2v}q^{\frac{r}{24}}.
\end{align*}
 Lemma \ref{lemma_Hecke-omega} establishes that the Hecke-type $q$-series in the second line of the above expression coincides with the generalized Appell-Lerch type series in Theorem \ref{theorem_main2}.

Applying the slash operator $\big|_{1+2v}S$ to $\pi_{\mathrm{hol}}([H_1,\eta]_v)$ and utilizing Proposition \ref{prop_holoproj-trans}, we conclude that
\[
-\zeta_4\pi_{\mathrm{hol}}([H_2,\eta]_v)=\pi_{\mathrm{hol}}([H_1,\eta]_v)\big|_{2v+1}S
\]
is a modular form in $M_{2v+1}(S^{-1}\Gamma_0(8)S,\left(\frac{-4}{\cdot}\right))=M_{2v+1}(\Gamma^0(8),\left(\frac{-4}{\cdot}\right))$. The Eisenstein part of this modular form can be obtained by applying the slash operator $\big|_{1+2v}S$ to \eqref{equation_Eispart}. By Proposition \ref{proposition_Fricke-invo}, we find that this part is equal to
\[
\frac{1}{2}\left(-\frac{1}{E_{2v}}E_{2v+1}^{\chi_{-4},1}(\tau/4)+\frac{2^{-2v}}{E_{2v}}E_{2v+1}^{\chi_{-4},1}(\tau/8)+\frac{2^{-2v}}{E_{2v}}E_{2v+1}^{1,\chi_{-4}}(\tau/4)-\frac{2^{-2v}}{E_{2v}}E_{2v+1}^{1,\chi_{-4}}(\tau/8)\right).
\]
Finally, by Lemma \ref{lemma_Eis}, these Eisenstein series can be transformed into those Lambert series  stated in Theorem \ref{theorem_main2}, which completes the proof of Theorem \ref{theorem_main2}. 

\section{Some $q$-series lemmas}
In this section, we establish several $q$-series identities that are utilized throughout this paper.
\begin{lemma}\label{lemma_q0}
The following identity holds:
\begin{equation} \label{equation-lemma-q0}
\sum_{n\in \mathbb{Z}} \frac{(-1)^n q^{\frac{3n^2+n}{2}}}{1+q^n} = \sum_{n \in \mathbb{Z}} \left(\frac{12}{2n+1}\right) \frac{q^{\frac{n(n+1)}{6}}}{1+q^n}.
\end{equation}
\end{lemma}

\begin{proof}
Let $L(q)$ and $R(q)$ denote the left- and right-hand sides of \eqref{equation-lemma-q0}, respectively. We evaluate $R(q)$ by trisecting the summation index modulo 3. 
Observe that the Kronecker symbol $\left(\frac{12}{2n+1}\right)$ vanishes for $n \equiv 1 \pmod 3$, reduces to $(-1)^m$ for $n = 3m$, and to $-(-1)^m$ for $n = 3m-1$. Decomposing $R(q)$ over these residue classes yields:
\begin{equation} \label{equation-RHS_trisected}
R(q) := R_1+R_2= \sum_{m \in \mathbb{Z}} \frac{(-1)^m q^{\frac{3m^2+m}{2}}}{1+q^{3m}} + \sum_{m \in \mathbb{Z}} \frac{(-1)^m q^{\frac{3m^2-m}{2}}}{1+q^{3m-1}}.
\end{equation}

For $L(q)$, we multiply the summand by $(1 - q^n + q^{2n})/(1 - q^n + q^{2n})$, which splits the series into three parts: $L(q)=L_1-L_2+L_3$, where $L_1=\sum_{n \in \mathbb{Z}} (-1)^n\frac{ q^{\frac{3n^2+n}{2}}}{1+q^{3n}}$. Note that $L_1$  coincides with $R_1$. For $L_3$, applying the involution $n \mapsto -n$, we find that $L_3=L_1$. Thus, it suffices to show that $R_1-R_2=L_2$.

Following Hickerson and Mortenson \cite{zbMATH06349729}, we define
\[
j(z;q) := (z;q)_\infty (q/z;q)_\infty (q;q)_\infty
\]
and
\begin{equation}\label{equation_m-function}
m(x,q,z) := \frac{1}{j(z;q)} \sum_{n \in \mathbb{Z}} \frac{(-1)^n q^{\binom{n}{2}} z^n}{1-q^{n-1}xz}.
\end{equation}
With this notation, we denote the part to be proven as
\begin{equation}\label{equation_mq-series}
j(q^2;q^3) m(-q, q^3, q^2)-j(q;q^3)m(-q,q^3,q)=j(q^3;q^3) m(-1, q^3, q^3).
\end{equation}
By the partial fraction expansion for the reciprocal of Jacobi’s theta product (see \cite[(1.3)]{zbMATH06349729})
\begin{equation}\label{equation_Jacobi-theta}
\sum_{n\in\mathbb{Z}}\frac{(-1)^nq^{n(n+1)/2}}{1-q^nz}=\frac{J_1^3}{j(z;q)},
\end{equation}
where $J_m:=j(q^{m},q^{3m})$, the right-hand side of \eqref{equation_mq-series} is equal to
\[
\frac{J_3^3}{j(-1;q^3)}.
\] 
For the left-hand side of \eqref{equation_mq-series}, we need the following result (see \cite[Theorem 2.3]{zbMATH06349729}):
\begin{equation} \label{equation_pole_shift}
m(x,q,z_1) - m(x,q,z_0) = \frac{z_0 J_1^3 j(z_1/z_0; q) j(x z_0 z_1; q)}{j(z_0;q) j(z_1;q) j(x z_0; q) j(x z_1; q)}.
\end{equation}
Substituting $q \mapsto q^3$, $x=-q$, $z_1=q^2$, and $z_0=q$ into \eqref{equation_pole_shift}, we obtain
\begin{equation}\label{equation_pole_shift2}
m(-q, q^3, q^2) - m(-q, q^3, q) = \frac{q J_3^3 j(q; q^3) j(-q^4; q^3)}{j(q; q^3) j(q^2; q^3) j(-q^2; q^3) j(-q^3; q^3)}.
\end{equation}
Apply the standard quasi-periodicity $j(q^3 z; q^3) = -z^{-1} j(z; q^3)$ and symmetry $j(z;q^3)=j(q^3/z;q^3)$, the right-hand side of \eqref{equation_pole_shift2} simplifies to:
\begin{equation*}
\frac{J_3^3}{J_1 j(-1; q^3)},
\end{equation*}
and this completes the proof.
\end{proof}

\begin{lemma}\label{lemma_q1}
The following identity holds:
\[
\sum_{n\in\mathbb{Z}}(-1)^nq^{3n(n+1)}\frac{1+q^{2n+1}}{1-q^{2n+1}}=-q^{-\frac{3}{4}}\sum_{n\in 2\mathbb{Z}+1}\left(\frac{12}{n+2}\right)\frac{q^{\frac{n(n+4)}{12}}}{1-q^{n}}
\]
\end{lemma}
\begin{proof}
First, by setting $n=2k+1$, the right-hand side series becomes:
\begin{equation*}
R(q) := -\sum_{k \in \mathbb{Z}} \left(\frac{12}{2k+3}\right) \frac{q^{\frac{k^2+3k-1}{3}}}{1-q^{2k+1}}.
\end{equation*}
Setting $k = 3m+1$ and $k = 3m-1$, we decompose $R(q) = R_1 - R_2$, where:
\[
R_1 = \sum_{m \in \mathbb{Z}} (-1)^m \frac{q^{3m^2+5m+1}}{1-q^{6m+3}}, ~~
R_2 = \sum_{m \in \mathbb{Z}} (-1)^m \frac{q^{3m^2+m-1}}{1-q^{6m-1}}.
\]

Using $\frac{1+q^{2n+1}}{1-q^{2n+1}} = \frac{2}{1-q^{2n+1}} - 1$ and  $\sum_{n \in \mathbb{Z}} (-1)^n q^{3n(n+1)} = 0$, we find that the left-hand side series becomes
\begin{equation*}
L(q) = \sum_{n \in \mathbb{Z}} \frac{(-1)^n q^{3n^2+3n}}{1-q^{2n+1}}.
\end{equation*}
We express the denominator as $(1+q^{2n+1}+q^{4n+2})/(1-q^{6n+3})$, which splits $L(q)$ into three sums: $L(q) = L_1 + L_2 + L_3$, where $L_1=\sum_{n\in\mathbb{Z}}(-1)^n\frac{q^{3n^2+3n}}{1-q^{6n+3}}$.
Notice $L_2$ identically matches $R_2$. Applying the involution $n \mapsto -n-1$ to $L_3$ precisely yields $L_2$. Thus, $L(q) = L_1 + 2R_1$. It then suffices to show $L_1 + R_1 + R_2 = 0$.

Using the $m$-function introduced in \eqref{equation_m-function}, we express $R_1$ and $R_2$ via the standard $m(x,q,z)$ function:
\begin{align*}
R_1 &= -q^{-1} j(q^2; q^6) m(q, q^6, q^2), \\
R_2 &= q^{-1} j(q^4; q^6) m(q, q^6, q^4) = q^{-1} j(q^2; q^6) m(q, q^6, q^4).
\end{align*}
Applying the pole-shifting identity \eqref{equation_pole_shift} for $z_1=q^4, z_0=q^2$:
\begin{equation*}
m(q, q^6, q^4) - m(q, q^6, q^2) = \frac{q^2 J_6^3 j(q^2; q^6) j(q^7; q^6)}{j(q^2; q^6) j(q^4; q^6) j(q^3; q^6) j(q^5; q^6)} = \frac{-q J_6^3}{j(q^2; q^6) j(q^3; q^6)}.
\end{equation*}
Substituting this back, we obtain
\begin{equation*}
R_1 + R_2 = q^{-1} j(q^2; q^6) \left[ \frac{-q J_6^3}{j(q^2; q^6) j(q^3; q^6)} \right] = -\frac{J_6^3}{j(q^3; q^6)} = -\frac{J_6^4}{J_3^2}.
\end{equation*}

Finally, by Jacobi’s theta product \eqref{equation_Jacobi-theta}, we evaluate $L_1$ using the standard Lambert series evaluation:
\begin{equation*}
L_1 = \sum_{n \in \mathbb{Z}} \frac{(-1)^n q^{3n^2+3n}}{1-q^{6n+3}} = \frac{J_6^4}{J_3^2}.
\end{equation*}
Therefore, $R_1 + R_2 + L_1 = \frac{J_6^4}{J_3^2} - \frac{J_6^4}{J_3^2} = 0$, completing the proof.
\end{proof}

\begin{lemma}\label{lemma_Hecke-f}
For any integer $v \ge 0$, the following identity holds:
\[
\sum_{n \in \mathbb{Z}} \left(\frac{12}{2n+1}\right) \frac{n^{2v} q^{n(n+1)/6}}{1+q^n} = \sum_{m > n > 0} \left(\frac{12}{m}\right)\left(\frac{-12}{n}\right) \left(\frac{m-n}{2}\right)^{2v} q^{\frac{m^2-n^2}{24}}.
\]
\end{lemma}

\begin{proof}
Let $L(q)$ denote the left-hand side of the equation. We separate $L(q)$ into positive and negative index domains: $L(q) = L_+ + L_-$. For $n > 0$, we expand the denominator as a geometric series $\frac{1}{1+q^n} = \sum_{j \ge 0} (-1)^j q^{jn}$, yielding:
\[
L_+ = \sum_{n \ge 1} \sum_{j \ge 0} (-1)^j \left(\frac{12}{2n+1}\right) n^{2v} q^{\frac{n(n+1+6j)}{6}}.
\]
We introduce a diagonalization of the quadratic form by setting $m = 2n + 6j + 1$ and $n' = 6j + 1$. Since $n \ge 1$ and $j \ge 0$, this maps bijectively to the domain $m > n' > 0$ with $n' \equiv 1 \pmod 6$. Under this substitution, we have $n = \frac{m-n'}{2}$ and the exponent becomes $\frac{m^2 - (n')^2}{24}$.
Furthermore, $n' \equiv 1 \pmod 6$ implies $\left(\frac{-12}{n'}\right) = 1$. The periodicity of the Kronecker symbol gives $\left(\frac{12}{m}\right) = \left(\frac{12}{2n+1+6j}\right) = (-1)^j \left(\frac{12}{2n+1}\right)$. Substituting these into $L_+$ yields:
\[
L_+ = \sum_{\substack{m > n' > 0 \\ n' \equiv 1 \pmod 6}} \left(\frac{12}{m}\right) \left(\frac{-12}{n'}\right) \left(\frac{m-n'}{2}\right)^{2v} q^{\frac{m^2-(n')^2}{24}}.
\]

For $n < 0$, we substitute $n = -u$ with $u > 0$. Similarly, the sum becomes
\[
L_- = \sum_{u \ge 1} \sum_{j \ge 1} (-1)^{j-1} \left(\frac{12}{2u-1}\right) u^{2v} q^{\frac{u(u-1+6j)}{6}}.
\]
We apply the substitution $m = 2u + 6j - 1$ and $n' = 6j - 1$. For $u, j \ge 1$, this maps to $m > n' > 0$ with $n' \equiv 5 \pmod 6$. 
Here, $n' \equiv 5 \pmod 6$ dictates $\left(\frac{-12}{n'}\right) = -1$. Thus, the negative index sum exactly generates the $n' \equiv 5 \pmod 6$ congruences:
\[
L_- = \sum_{\substack{m > n' > 0 \\ n' \equiv 5 \pmod 6}} \left(\frac{12}{m}\right) \left(\frac{-12}{n'}\right) \left(\frac{m-n'}{2}\right)^{2v} q^{\frac{m^2-(n')^2}{24}}.
\]
Since $\left(\frac{-12}{n'}\right)$ is supported entirely on $\gcd(n', 6) = 1$, the sum $L_+ + L_-$ perfectly reconstitutes the Hecke indefinite theta series over all valid coprime residue classes.
\end{proof}

\begin{lemma}\label{lemma_Hecke-omega}
For every integer $v\geq 0$, the following identity holds:
\[
\sum_{n\in \mathbb{Z}}\left(\frac{12}{n+2}\right)\frac{n^{2v}q^{n(n+4)/24}}{1-q^{n/2}}=\sum_{m>n>0}\bigg(\frac{12}{m}\bigg) \bigg(\frac{n}{3}\bigg)(-1)^{n-1}\left(m-2n\right)^{2v}q^{\frac{m^2-4n^2}{24}}.
\]
\end{lemma}
\begin{proof}
The proof is entirely analogous to that of Lemma \ref{lemma_Hecke-f}. We establish the identity by expanding the denominator geometrically and splitting the summation into positive and negative domains of $n$. For $n > 0$, expanding $(1 - q^{n/2})^{-1} = \sum_{j \ge 0} q^{jn/2}$ and applying the substitution $m = n + 6j + 2$ and $n' = 3j + 1$ bijectively generates the terms of the Hecke double sum restricted to the residue class $n' \equiv 1 \pmod 3$. For $n < 0$, we set $n = -u$ and expand $(1 - q^{-u/2})^{-1} = -\sum_{j \ge 1} q^{ju/2}$. The substitution $m = u + 6j - 2$ and $n' = 3j - 1$ similarly recovers exactly the terms corresponding to $n' \equiv -1 \pmod 3$. 

One can straightforwardly verify that the quadratic forms factorize as $m^2 - 4n'^2 = n(n+4)$  in both domains. The remaining character evaluations follow from the periodicity of the Kronecker symbol exactly as demonstrated in Lemma \ref{lemma_Hecke-f}.
\end{proof}

\begin{lemma}\label{lemma_Eis}
For every $v\geq 0$, the following identities hold:
\begin{align*}
&E_{2v+1}^{1,\chi_{-4}}(\tau)-2E_{2v+1}^{1,\chi_{-4}}(2\tau)+\frac{E_{2v}}{2}=2\sum_{n\geq 1}\frac{(-1)^{n-1}(2n-1)^{2v}q^{2n-1}}{1+q^{2n-1}},\\
&E_{2v+1}^{1,\chi_{-4}}(\tau)-E_{2v+1}^{1,\chi_{-4}}(2\tau)=2\sum_{n\geq 1}\frac{(-1)^{n-1}(2n-1)^{2v}q^{2n-1}}{1-q^{4n-2}},\\
&E_{2v+1}^{\chi_{-4},1}(\tau)-2^{2v+1}E_{2v+1}^{\chi_{-4},1}(2\tau)=2\sum_{n\geq 1}\frac{(-1)^{n-1}n^{2v}q^n}{1+q^n},\\
&E_{2v+1}^{\chi_{-4},1}(\tau)-2^{2v}E_{2v+1}^{\chi_{-4},1}(2\tau)=2\sum_{n\geq 1}\frac{(2n-1)^{2v}q^{2n-1}}{1+q^{4n-2}}.
\end{align*}
\end{lemma}
\begin{proof}
It is well-known that
\[
\sum_{n\geq 1}\sum_{\substack{d\mid n\\d>0}}\chi_{-4}(d)d^{s}q^n=\sum_{n\geq 1}\frac{(-1)^{n-1}(2n-1)^sq^{2n-1}}{1-q^{2n-1}}
\]
and
\[
\sum_{n\geq 1}\sum_{\substack{d\mid n\\d>0}}\chi_{-4}(n/d)d^{s}q^n=\sum_{n\geq 1}\frac{n^sq^{n}}{1+q^{2n}}.
\]
The result follows from elementary $q$-series analysis. For instance, the first identity  is equivalent to
\begin{align*}
&\frac{(-1)^{n-1}(2n-1)^sq^{2n-1}}{1-q^{2n-1}}-\frac{2(-1)^{n-1}(2n-1)^sq^{4n-2}}{1-q^{4n-2}}\\
&=\frac{(-1)^{n-1}(2n-1)^sq^{2n-1}(1+q^{2n-1})}{(1-q^{2n-1})(1+q^{2n-1})}-\frac{2(-1)^{n-1}(2n-1)^sq^{4n-2}}{1-q^{4n-2}}\\
&=\frac{(-1)^{n-1}(2n-1)^{2v}q^{2n-1}}{1+q^{2n-1}}
\end{align*}
The proofs of the remaining identities are left to the reader.
\end{proof}

\providecommand{\bysame}{\leavevmode\hbox to3em{\hrulefill}\thinspace}
\providecommand{\MR}{\relax\ifhmode\unskip\space\fi MR }
\providecommand{\MRhref}[2]{%
  \href{http://www.ams.org/mathscinet-getitem?mr=#1}{#2}
}
\providecommand{\href}[2]{#2}

\end{document}